\documentclass[12pt,a4paper]{amsart}

\usepackage[english]{babel}
\usepackage[T1]{fontenc}
\usepackage[utf8]{inputenc}

\usepackage{amsmath, amsthm, amscd, amsfonts}
\usepackage{amssymb, tikz}
\usepackage{mathtools}
\usepackage{mathrsfs}
\usepackage{hyperref}
\usepackage[capitalise]{cleveref}
\usepackage{yfonts}
\usepackage{setspace}
\newcommand{\id}{\text{id}}

\DeclareMathOperator{\len}{lh}

\DeclareMathOperator{\dom}{dom}
\DeclareMathOperator{\cf}{cf}

\DeclareMathOperator{\CH}{CH}

\DeclareMathOperator{\range}{range}

\DeclareMathOperator{\ZFC}{ZFC}
\DeclareMathOperator{\MA}{MA}

\DeclareMathOperator{\AP}{AP}

\def\MPB{{\mathbb{P}}}
\def\MQB{{\mathbb{Q}}}

\def\MMB{{\mathbb{M}}}
\def\MNB{{\mathbb{N}}}
\def\MKB{{\mathbb{K}}}

\newcommand{\cH}{{\mathscr H}}

\newtheorem{theorem}{Theorem}[section]
\newtheorem{lemma}[theorem]{Lemma}
\newtheorem{proposition}[theorem]{Proposition}
\newtheorem{corollary}[theorem]{Corollary}
\newtheorem{notation}[theorem]{Notation}

\newtheorem{convention}[theorem]{Convention}

\theoremstyle{definition}
\newtheorem{definition}[theorem]{Definition}

\theoremstyle{remark}
\newtheorem{remark}[theorem]{Remark}
\newtheorem{claim}[theorem]{Claim}

\title[The Keisler--Shelah isomorphism theorem and the continuum hypothesis II]{The Keisler--Shelah isomorphism theorem and the continuum hypothesis II}

\author[M. Golshani and S. Shelah]{Mohammad  Golshani and Saharon Shelah}
\thanks{%
 The first author's research has been supported by a grant from
  IPM (No. 1400030417). The
	second author's research  partially supported by
ISF 1838/19: The Israel Science Foundation (ISF) (2019-2023) and
Rutgers 2018 DMS 1833363: NSF DMS Rutgers visitor program (PI S. Thomas) (2018-2022).
This is publication 1223 of the second author. The first author thanks  Mostafa Mirabi and Rahman Mohammadpour for carefully reading the paper and providing useful suggestions. The authors thank the referee of the paper for many useful comments and suggestions.}

\address{
School of Mathematics, Institute for Research in Fundamental Sciences (IPM), P.O. Box:
19395-5746, Tehran-Iran. }
\email{golshani.m@gmail.com}

\address{
Einstein Institute of Mathematics, The Hebrew University of Jerusalem, Jerusalem,
	91904, Israel, and Department of Mathematics, Rutgers University, New Brunswick, NJ
	08854, USA.}
\email{shelah@math.huji.ac.il}

\subjclass[2010]{03C20, 03E35}

\keywords{}

\begin{document}
\begin{abstract}
We continue the investigation started in \cite{gol-sh} about the relation between the Keilser--Shelah isomorphism theorem and the continuum hypothesis.
In particular, we show it is consistent that the continuum hypothesis fails and for any given  sequence $\mathbf m=\langle (\MMB^{1}_n, \MMB^{2}_n): n < \omega              \rangle$ of models of size at most $\aleph_1$ in a countable language, if the sequence satisfies a mild extra property, then
for every non-principal ultrafilter $\mathcal D$ on $\omega$, if the ultraproducts  $\prod\limits_{\mathcal D} \MMB^{1}_n$ and  $\prod\limits_{\mathcal D} \MMB^{2}_n$
are elementarily equivalent, then they  are isomorphic.

\end{abstract}
 \maketitle


\section{Introduction}
Ultraproducts arise naturally in  model theory and many other areas of mathematics, see \cite[Chapter VI]{sh:c}. An ultraproduct is a way to connect the notions of elementary equivalence and isomorphism. By a result
of Keisler \cite{keislr}, the continuum hypothesis,  $\CH$, implies that in a countable language $\mathcal L$, two $\mathcal L$-models  $\bold M, \bold N$ of size $\leq 2^{\aleph_0}$,
are elementarily equivalent if and only if they have isomorphic ultrapowers with respect to an ultrafilter on $\omega.$  Recently the authors of this paper \cite{gol-sh} have shown that Keisler's theorem is indeed equivalent to the $\CH$, by showing that there are two elementary equivalent dense linear orders $M$ and  $N$ of size $\leq \aleph_2$ which do not have isomorphic ultrapowers with respect to any ultrafilter on $\omega.$
 Much earlier but after Keisler, Shelah \cite{sh:iso} removed  the $\CH$ from Keisler's theorem by weakening the conclusion  and showed that if $\mathcal L$ is a countable language  and $\bold M, \bold N$ are  countable $\mathcal L$-models,
then $\bold M \equiv \bold N$ if and only if
 they have isomorphic ultrapowers with respect to an ultrafilter on
  $2^\omega$.
 Shelah \cite{sh326} has shown that
the $\CH$ is an essential assumption for Keisler's theorem, even for countable models, by constructing a model of $\ZFC$ in which $2^{\aleph_0}=\aleph_2$
and  there are countable graphs $\bold \Delta \equiv \bold \Gamma$  such that for no ultrafilter $\mathcal U$
on $\omega,$  $\bold \Delta^{\omega}/ \mathcal U \simeq \bold \Gamma^\omega / \mathcal U.$
See also \cite{goto}, where some further connections  between several variants of Keilser's theorem and cardinal invariants are found.

In this paper, we continue the investigations started in \cite{gol-sh}.
 In Section \ref{etkit}, we consider some further extensions of Keisler's isomorphism theorem.
 To this end, we define the notion of an ultraproduct problem
 $\bold{m}=\langle (\MMB^{\bold m, 1}_n, \MMB^{\bold m, 2}_n, \tau_{\bold m, n}): n < \omega              \rangle$ and  show the consistency of the failure of the $\CH$ with the assertion that for any non-principal ultrafilter $\mathcal D$ on $\omega,$ if
 the ultraproducts
 $\prod\limits_{\mathcal D} \MMB^{\bold m, 1}_n$ and $\prod\limits_{\mathcal D} \MMB^{\bold m, 2}_n$ are elementarily equivalent, then they are indeed isomorphic.
In  Section \ref{keisler2} we show that the above conclusion fails if $\mathfrak{b}> \aleph_1$, where $\mathfrak{b}$ denotes the bounding number.
\section{Extending the Keisler isomorphism theorem}
\label{etkit}
Recall from \cite{gol-sh} that there exists a dense linear order $N$ of size $\aleph_2$ which is elementary equivalent to $M=(\mathbb{Q}, <)$,
but for no ultrafilter $\mathcal U$ on $\omega, {}^{\omega}{M}/\mathcal U \simeq {}^{\omega}{N}/\mathcal U$. We also proved several consistent
extensions of the Keisler's theorem
for models of size at most $\aleph_1$ in the absence of the continuum hypothesis.
 In this section, we continue the work started in \cite{gol-sh} and give a further extension of Keisler's isomorphism theorem.
To this end, we start by making some definitions.
\begin{definition}
\label{upp}
A {\em pseudo ultraproduct problem} is a sequence
\[
\bold{m}=\langle (\MMB^{\bold m, 1}_n, \MMB^{\bold m, 2}_n, \tau_{\bold m, n}): n < \omega              \rangle
\]
where
\begin{enumerate}
\item $\langle \tau_{\bold m, n}: n<\omega  \rangle$ is a $\subseteq$-increasing sequence of finite vocabularies, with $\tau_{\bold m, 0}=\emptyset$. Set $\tau_{\bold m}=\bigcup\limits_n \tau_{\bold m, n},$\footnote{Thus we allow that $\tau_{\bold m}$ to be finite.}

\item each $\MMB^{\bold m, \ell}_n$ is a $\tau_{\bold m}$-model,
\item $\kappa(\bold{m}) \leq \aleph_1,$ where $\kappa(\bold{m})=\sup\{ ||\MMB^{\bold m, \ell}_n||: \ell=1, 2,$ and $ n<\omega        \}$
and $ ||\MMB^{\bold m, \ell}_n||$ denotes the size of the universe of the model $\MMB^{\bold m, \ell}_n$.

\end{enumerate}
\end{definition}
\begin{definition}
\label{ehfrgame0} Suppose $\bold m$ is a pseudo ultraproduct problem and $k \leq n<\omega.$
\begin{enumerate}
\item The Ehrenfeucht–Fra\"{i}ss\'{e} game $\Game^n_k(\bold m)=\Game_{\tau_{\bold m, k}, k}(\MMB^{\bold m, 1}_n, \MMB^{ \bold m, 2}_n)$
is defined as a game between two players protagonist and
antagonist  where
\begin{enumerate}
\item[(a)] it has $2(k+1)$ moves,

\item[(b)] protagonist plays at even stages and  antagonist plays at odd stages,

\item[(c)] in the $(2l+1)$-th move, the antagonist chooses  $A_l \subseteq \MMB^{\bold m, 1}_n, B_l \subseteq \MMB^{\bold m, 2}_n$
such that $|A_l|+|B_l| \leq k$,

\item[(d)] in the $(2l+2)$-th move, the protagonist chooses $f_l$, a partial one-to-one function  from $\MMB^{\bold m, 1}_n \restriction \tau_{ \bold m, k}$ into $\MMB^{\bold m, 2}_n$, which preserves $\phi$ and $\neg\phi,$ for $\phi$ a strictly atomic formula (i.e., $\phi$ is of the form $x=y$ or $P(x_0, \cdots, x_{m-1})$ or $F(x_0, \cdots, x_{m-1})=y,$ where $P$ is a predicate symbol and $F$ is a function symbol or an individual constant) from $\tau_{ \bold m, k}$,

\item[(e)] the protagonist has to  satisfy $A_l \subseteq \dom(f_l)$, $B_l  \subseteq \range(f_l)$
and $f_l \supseteq f_{l-1}$,
\end{enumerate}

\item We say that the protagonist looses the game $\Game^n_k(\bold m)$, when there is no legal move for him to do.
\end{enumerate}
\end{definition}
The following easy lemma will be useful later.
\begin{lemma}
\label{ss1}
Suppose $\bold m$ is a pseudo ultraproduct problem  and $k \leq n<\omega.$ Let
$f$ be the last move of protagonist in the game $\Game^n_k(\bold m)$. If $\phi(\nu_0, \cdots, \nu_{l-1})$
is a $\tau_{ \bold m, k}$-formula and $x_0, \cdots, x_{l-1} \in \dom(f)$,
then
\[
\MMB^{\bold m, 1}_n \models \phi_{\dom(f)}(x_0, \cdots, x_{l-1}) \Leftrightarrow \MMB^{\bold m, 2}_n \models \phi_{\range(f)}(f(x_0), \cdots, f(x_{l-1})),
\]
where for any set $D$, $\phi_D$ is obtained from $\phi$ by replacing all quantifiers $\exists x$ and $\forall x$ by the restricted quantifiers $\exists x \in D$
and $\forall x \in D$ respectively.
\end{lemma}
\begin{proof}
By induction on the complexity of the formula $\phi$. This is true for strictly atomic formulas by the assumption
and it is easy to see that if it holds for  $\phi, \phi_0$ and  $\phi_1$, then it also holds for $\neg\phi$ and $\phi_0 \wedge \phi_1$.
Now suppose that $\phi(\nu_0, \cdots, \nu_{l-1})=\exists \nu \psi(\nu, \nu_0, \cdots, \nu_{l-1})$
and let $x_0, \cdots, x_{l-1} \in \dom(f)$. If $x \in \dom(f)$ is such that
$\MMB^{\bold m, 1}_n \models \psi_{\dom(f)}(x, x_0, \cdots, x_{l-1})$, then by the  induction hypothesis
$\MMB^{\bold m, 2}_n \models \psi_{\range(f)}(f(x), f(x_0), \cdots, f(x_{l-1}))$ and hence $\MMB^{\bold m, 2}_n \models \exists \nu \in \range(f) \psi_{\range(f)}(\nu, f(x_0), \cdots, f(x_{l-1}))$. It then follows that $\MMB^{\bold m, 2}_n \models \phi_{\range(f)}(f(x_0), \cdots, f(x_{l-1}))$. Conversely suppose that for some $y \in \range(f)$,  $\MMB^{\bold m, 2}_n \models \psi_{\range(f)}(y, f(x_0), \cdots, f(x_{l-1}))$. Let $x \in \dom(f)$ be such that $y=f(x)$. Then by the induction hypothesis
$\MMB^{\bold m, 1}_n \models \psi_{\dom(f)}(x, x_0, \cdots, x_{l-1})$ and hence $\MMB^{\bold m, 1}_n \models \exists \nu \in \dom(f)\psi_{\dom(f)}(\nu, x_0, \cdots, x_{l-1})$. Thus $\MMB^{\bold m, 1}_n \models \phi_{\dom(f)}(x_0, \cdots, x_{l-1})$.
\end{proof}
\begin{definition}
\label{ehfrgame}
Suppose $\bold m$ is a pseudo ultraproduct problem and $n<\omega.$
Then $\bold{k}_{\bold{m}, n}$ is the maximal $k \leq n$ such that the protagonist
has a winning strategy in the Ehrenfeucht–Fra\"{i}ss\'{e} game $\Game^n_k(\bold m)=\Game_{\tau_{\bold m, k}, k}(\MMB^{\bold m, 1}_n, \MMB^{ \bold m, 2}_n)$.
Set also $\bold{k}_{\bold{m}}=\langle \bold{k}_{\bold{m}, n}: n<\omega \rangle.$
\end{definition}
We now define the notion of an ultraproduct problem.
\begin{definition}
\label{a2}
An {\em ultraproduct problem} is a pseudo ultraproduct problem $\bold m$ such that
$\limsup\limits_{n<\omega}\bold{k}_{\bold m, n}=\infty.$
\end{definition}
To each pseudo ultraproduct problem we  assign a natural countably generated filter on $\omega$ and a cardinal invariant, which play
an important role for the rest of the paper.
We start by defining such notions in a more general context. Let us first fix some notation.
\begin{notation}
(1) For a filter $\mathcal D$ on $\omega,$ let $\forall_{\mathcal D}x \phi(x)$ mean ``$\exists A \in \mathcal{D} ~\forall n \in A ~\phi(n)$''.
\\
(2) The notation $\forall^* x \phi(x)$ means ``for all but finitely many $x, \phi(x)$ holds''.
\end{notation}
\begin{definition}
\label{g32}
\begin{enumerate}
\item Given a sequence $\bold k=\langle \bold k_n: n<\omega \rangle \in {}^\omega{\omega}$, let $\mathcal D_{\bold k}$ be the filter on
 $\omega$ generated by co-bounded subsets of $\omega$
and the sets
\[
\{ n<\omega: \bold{k}_n > k           \},
\]
where $k<\omega.$

\item Suppose $\bold k=\langle \bold k_n: n<\omega \rangle \in {}^\omega{\omega}$ is such that $\limsup\limits_{n<\omega} \bold k_n=\infty,$ then set
  \begin{center}
  $\bold{\mathfrak d}_{\bold k}=\min\bigg\{|\mathcal{F}|: \mathcal{F} \subseteq \prod\limits_{n<\omega}[\omega]^{\bold{k}_n}$ and $\left(\forall \eta \in {}^\omega{\omega}\right)\left(  \exists f \in \mathcal{F}\right)\big( \forall_{\mathcal{D}_{\bold k}}n \left(\eta(n)\in f(n)\right)\big)   \bigg\}$.
  \end{center}
\end{enumerate}
\end{definition}
\begin{remark}
Suppose $\bold k=\langle \bold k_n: n<\omega \rangle \in {}^\omega{\omega}$.
If $\limsup\limits_{n<\omega} \bold k_n=\infty,$ then
 $\mathcal{D}_{\bold k}$ is a countably generated non-principal proper filter on $\omega$. Otherwise, $\mathcal{D}_{\bold k}=\mathcal{P}(\omega)$.
 \end{remark}

\begin{definition}
\label{filter}
 If $\bold m$ is a pseudo ultraproduct problem, then  set $\mathcal{D}_{\bold{m}}=\mathcal D_{\bold k_{\bold m}}$. Furthermore, if
 it is  an ultraproduct problem, then set $\bold{\mathfrak d}_{\bold m}=\bold{\mathfrak d}_{\bold k_{\bold m}}.$
\end{definition}

The next simple lemma will be very useful.
\begin{lemma}
\label{eqv3}
Suppose  $\bold m$ in an ultraproduct problem.
\begin{enumerate}
\item If $\mathcal D \supseteq \mathcal D_{\bold m}$ is a non-principal ultrafiler on $\omega$.
Then the ultraproducts
$\prod\limits_{\mathcal D} \MMB^{\bold m, 1}_n$ and $\prod\limits_{\mathcal D} \MMB^{\bold m, 2}_n$
are elementary equivalent.

\item  If $\mathcal D$ is a non-principal ultrafiler on $\omega$ and
$\prod\limits_{\mathcal D} \MMB^{\bold m, 1}_n \equiv \prod\limits_{\mathcal D} \MMB^{\bold m, 2}_n$,
then $D \supseteq \mathcal D_{\bold m}$.
\end{enumerate}
\end{lemma}
\begin{proof}
(1) Suppose $\bold m$ is an ultraproduct problem
and $D \supseteq \mathcal D_{\bold m}$ is a non-principal ultrafiler on $\omega$. Let $\phi$ be a $\tau_{\bold m}$-statement.
Then, for some $k<\omega$,
it is a $\tau_{\bold m, k}$-statement.
We prove by induction on the complexity of $\phi$ that
\[
(*)_\phi: \qquad \prod\limits_{\mathcal D} \MMB^{\bold m, 1}_n \models \phi \iff \prod\limits_{\mathcal D} \MMB^{\bold m, 2}_n \models \phi.
\]
If $\phi$ is a strictly atomic formula, then
$E=\{ n< \omega: \bold k_{\bold m, n} > k    \} \in \mathcal D_{\bold m} \subseteq \mathcal D,$ and for each
$n \in E$ we have $\MMB^{\bold m, 1}_n \models \phi$ if and only if $\MMB^{\bold m, 2}_n \models \phi$,
from which the result follows. It is also clear that if $(*)_{\phi}$
holds then $(*)_{\neg\phi}$ holds and that if $(*)_{\phi_0}$ and $(*)_{\phi_1}$ hold, then $(*)_{\phi_0 \wedge \phi_1}$ holds.
Now suppose that $\phi=\exists x \psi(x)$ and $(*)_{\psi}$ is true.
Suppose $\prod\limits_{\mathcal D} \MMB^{\bold m, 1}_n \models \phi$. Then for some $\bar x=\langle   x_n: n<\omega   \rangle \in \prod_{n<\omega}\MMB^{\bold m, 1}_n,$
$\prod\limits_{\mathcal D} \MMB^{\bold m, 1}_n \models \psi([\bar x]_{\mathcal D})$, and hence
\[
A=\{ n<\omega:  \MMB^{\bold m, 1}_n \models \psi(x_n)                \} \in \mathcal D.
\]
We may further suppose that for any $n \in A, \bold k_{\bold m, n} > k.$
For $n\in A$ let $g_n$ be the last move of the  protagonist in the game $\Game^n_{\bold k_{\bold m, n}}(\bold m)$,
in which antagonist always chooses $A_\ell=\{x_n \}$ and $B_\ell=\emptyset$.
Thus for any such $n$, by the induction hypothesis and Lemma \ref{ss1} we have
$\MMB^{\bold m, 2}_n \models \psi(g_n(x_n)))$. Let $\bar y=\langle  g_n(x_n): n<\omega      \rangle$. Then
$\prod\limits_{\mathcal D} \MMB^{\bold m, 2}_n \models \psi([\bar y]_{\mathcal D})$ and hence $\prod\limits_{\mathcal D} \MMB^{\bold m, 2}_n \models \phi$. Conversely, suppose that $\prod\limits_{\mathcal D} \MMB^{\bold m, 2}_n \models \phi$.
Then for some $\bar y= \langle   y_n: n<\omega  \rangle \in \prod_{n<\omega}\MMB^{\bold m, 2}_n,$
$\prod\limits_{\mathcal D} \MMB^{\bold m, 2}_n \models \psi([\bar y]_{\mathcal D})$, and hence
\[
B=\{ n<\omega:  \MMB^{\bold m, 2}_n \models \psi(y_n)                \} \in \mathcal D.
\]
We again assume that $\bold k_{\bold m, n} > k$ for every $n \in B.$
For $n\in B$, let $h_n$ be the last move of the  protagonist in the game $\Game^n_{\bold k_{\bold m, n}}(\bold m)$,
in which antagonist always chooses $A_\ell=\emptyset$ and $B_\ell=\{y_n \}$.
Thus for any such $n$, $y_n \in \range(h_n)$ and hence for some $x_n, y_n=h_n(x_n)$.  By the induction hypothesis and Lemma \ref{ss1} we have
$\MMB^{\bold m, 1}_n \models \psi(x_n)$. Let $\bar x=\langle x_n: n<\omega      \rangle$. Then
$\prod\limits_{\mathcal D} \MMB^{\bold m, 1}_n \models \psi([\bar x]_{\mathcal D})$ and hence $\prod\limits_{\mathcal D} \MMB^{\bold m, 1}_n \models \phi$. We are done.

(2) Suppose by the way of contradiction that $D \nsupseteq \mathcal D_{\bold m}$. Let $k<\omega$ be such that $\{ n<\omega:  \bold k_{\bold m, n} > k           \} \in \mathcal D_{\bold m} \setminus D$. It then follows that
\[
A=\{ n<\omega:  \bold k_{\bold m, n} \leq  k           \} \in D.
\]

Thus for any $n \in A$ with $n>k$, as $ \bold k_{\bold m, n} < k+1,$  the protagonist loses the game $\Game^n_{k+1}(\bold m)$, and hence antagonist has a winning strategy.  In particular, by enlarging $k$ if necessary, we can find some  formula $\phi_n(\nu_0, \cdots, \nu_{l_n-1})$, which is a boolean combination of strict atomic formulas of $\tau_{\bold m, k+1}$, and some $a_0, \cdots, a_{l_n-1} \in \MMB^{\bold m, 1}_n$ such that $ \MMB^{\bold m, 1}_n \models \phi_n (a_0, \cdots, a_{l_n-1})$,
but for no $b_0, \cdots, b_{l_n-1} \in \MMB^{\bold m, 2}_n$ we have $ \MMB^{\bold m, 2}_n \models \phi_n(b_0, \cdots, b_{l_n-1})$.
It thus follows that
$ \MMB^{\bold m, 1}_n \models \exists x_0 \cdots x_{l_n-1}\phi_n$
but $\MMB^{\bold m, 2}_n \models \neg\exists x_0 \cdots x_{l_n-1}\phi_n$.
As $\tau_{\bold m, k+1}$ is finite and $\mathcal D$ is an ultrafiler, for some set $B \subseteq A$ in $\mathcal D$ we have $\phi_n=\phi$ for some fixed formula $\phi$. But then
$\prod\limits_{\mathcal D} \MMB^{\bold m, 1}_n \models \exists x_0 \cdots x_{l_n-1}\phi$ while  $\prod\limits_{\mathcal D} \MMB^{\bold m, 1}_n \models \neg\exists x_0 \cdots x_{l_n-1}\phi,$ which contradicts our assumption $\prod\limits_{\mathcal D} \MMB^{\bold m, 1}_n \equiv \prod\limits_{\mathcal D} \MMB^{\bold m, 2}_n$.
\end{proof}

We would like to construct a model of $\ZFC$ in which the continuum is large and
for every pseudo ultraproduct problem $\bold m$, if $\mathcal D$ is a non-principal ultrafilter on $\omega$
and if the structures $\prod\limits_{\mathcal D} \MMB^{\bold m, 1}_n$ and $\prod\limits_{\mathcal D} \MMB^{\bold m, 2}_n$ are elementarily equivalent, then
  $\prod\limits_{\mathcal D} \MMB^{\bold m, 1}_n \cong \prod\limits_{\mathcal D} \MMB^{\bold m, 2}_n$.
  We may assume that $\bold m$ is an ultraproduct problem, as otherwise $\mathcal D_{\bold m}=\mathcal P(\omega)$,
  and everything becomes trivial.

The next lemma reduces the construction of such a model to controlling the size of
$\bold{\mathfrak d}_{\bold m}$'s.

\begin{proposition}
\label{f207} Suppose $\bold m$ is an ultraproduct problem, $\mathfrak{d}_{\bold m}=\aleph_1$ and $\mathcal D$ is a non-principal ultrafilter on $\omega$ such that $\prod\limits_{\mathcal D} \MMB^{\bold m, 1}_n$ and  $\prod\limits_{\mathcal D} \MMB^{\bold m, 2}_n$
are elementarily equivalent. Then these ultraproducts are isomorphic, i.e., $\prod\limits_{\mathcal D} \MMB^{\bold m, 1}_n \cong \prod\limits_{\mathcal D} \MMB^{\bold m, 2}_n$.
\end{proposition}
\begin{remark}
By \cite{sh326}, it is consistent that $\mathfrak d =\aleph_1 < \bold{\mathfrak d}_{\bold m}$, where $\mathfrak d$ is the dominating number. It also follows from \cite{sh326} that in
Lemma \ref{f207}, we can not replace $\mathfrak{d}_{\bold m}=\aleph_1$ by $\mathfrak d =\aleph_1$.
\end{remark}
The following definition plays a key role in the proof of Lemma \ref{f207}.
\begin{definition}
\label{app}
Suppose $\bold m$ is an ultraproduct problem.
\begin{enumerate}
\item We say that $\bold s \in \AP_{\bold m}$ iff
\begin{enumerate}
\item $\bold s= \langle \bold g_{\bold s, n}: n<\omega      \rangle$,

\item $\bold g_{\bold s, n}$ is an initial segment of a play of $\Game^n_{\bold k_{\bold m, n}}(\bold m)$ of length $l_{\bold s, n}$ with the last function
$f_{\bold s, n}$, where  the protagonist plays with a winning strategy,
\item $\lim_{\mathcal{D}_{\bold m}}\langle \bold{k}_{\bold m, n} - l_{\bold s, n}: n<\omega \rangle=\infty.$
\end{enumerate}
\item Define the partial order $\leq_{\AP_{\bold m}}$ on $\AP_{\bold m}$ by $\bold s \leq_{\AP_{\bold m}}  \bold t$ iff
\[
\{n: \bold g_{\bold s, n} \text{~is an initial segment of~} \bold g_{\bold t, n}      \} \in \mathcal{D}_{\bold m}.
\]
\end{enumerate}
\end{definition}
We are now ready to complete the proof of Lemma \ref{f207}.
\begin{proof}[Proof of Lemma \ref{f207}]
Assume the hypotheses in the lemma hold.
 The proof of the next claim is evident.
 \begin{claim}
\label{f201}
\begin{enumerate}
\item $\AP_{\bold m} \neq \emptyset.$

\item If $\langle \bold s_\ell: \ell<\omega        \rangle$ is a $\leq_{\AP_{\bold m}}$-increasing sequence from $\AP_{\bold m},$ then it has a $\leq_{\AP_{\bold m}}$-upper bound.

\item Suppose $\bold s \in \AP_{\bold m}$,  $\ell \in \{1,2\}$, and $\bar w=\langle  w_n: n<\omega    \rangle \in \prod\limits_n [\MMB^{\bold m, \ell}_n]^{\bold k_{\bold m, n}}$. Then for some $\bold t \in \AP_{\bold m},$ we have
    \begin{enumerate}
    \item $\bold s \leq_{\AP_{\bold m}} \bold t,$
    \item if $l_{\bold s, n} < \bold k_{\bold m, n},$ then
    \begin{enumerate}
    \item $\ell=1 \implies w_n \subseteq \dom(f_{\bold t, n})$,
    \item $\ell=2 \implies w_n \subseteq \range(f_{\bold t, n})$.
    \end{enumerate}
    \end{enumerate}
\end{enumerate}
\end{claim}
\begin{proof}
 Clause (1) is clear. Clause (2) follows by an easy diagonalization argument, but let us elaborate a proof.
For each $\ell < \omega$ set
\[
\eta_\ell= \langle \bold k_{\bold m, n} - l_{\bold s_\ell, n}: n<\omega        \rangle \in{} {\omega}^{\omega}.
\]
 Then $\lim_{\mathcal D_{\bold m}}\eta_\ell(n)=\infty$, and for all $i<\omega,$
 \[
 (\omega, <)^{\omega} / \mathcal D_{\bold m} \models  \id_i /\mathcal D_{\bold m} < \eta_\ell / \mathcal D_{\bold m},
 \]
 where $\id_i$ is the constant sequence $i$ on $\omega.$
 As the structure
 $(\omega, <)^{\omega} / \mathcal D_{\bold m}$ is $\aleph_1$-saturated, we can find some $\eta \in {}{\omega}^{\omega}$
 such that for all $\ell, i<\omega$,
 \[
  (\omega, <)^{\omega} / \mathcal D_{\bold m} \models  \id_i /\mathcal D_{\bold m} <  \eta/\mathcal D_{\bold m}   <\eta_\ell / \mathcal D_{\bold m}.
 \]
 For $n<\omega$ let $\ell_n$ be the maximal natural number $\ell$ such that:
 \begin{enumerate}
 \item[(a)] $\ell \leq \eta(n),$

 \item[(b)] $\langle g_{\bold s_i, n}: i \leq \ell+1     \rangle$ is $\unlhd$-increasing,

 \item[(c)] $\bold k_{\bold m, n} -  l_{\bold s_i, n} \geq \eta(n)$ for every $i \leq \ell.$
 \end{enumerate}
 Note that each $\ell_n$ is well-defined as on the one hand, $\ell_n \leq \eta(n),$
 and on the other hand, $\ell=0$ satisfies the above requirements.

 Let $\ell_* < \omega.$ Then
 \[
 D_{\ell_*}=\{n<\omega: \ell_n \geq \ell_*,  k_{\bold m, n} -  l_{\bold s_{\ell(n)}, n} \geq \eta(n)  \} \in \mathcal D_{\bold m}.
 \]
 To see this, note that the set
 \[
 \{ n<\omega: \forall j \leq \ell_* \big(\eta(n) > \ell_*, \bold g_{\bold s_j, n} \unlhd \bold g_{\bold s_{j+1}, n} \text{~and~}  k_{\bold m, n} -  l_{\bold s_j, n} \geq \eta(n)      \big) \}
 \]
belongs to $\mathcal D_{\bold m}$ and is included in $D_{\ell_*},$ so $D_{\ell_*} \in \mathcal D_{\bold m}$ as well.

 Now set
 $\bold s= \langle \bold g_{\bold s, n}: n<\omega      \rangle$, where $\bold g_{\bold s, n}=\bold g_{\bold s_{\ell_n}, n}$ for each $n<\omega$.
We show that $\bold s$ is as required.

In order to show that $\bold s \in \AP_{\bold m}$, it only
 suffices to show that
$$\lim_{\mathcal D_{\bold m}} \bold k_{\bold m, n} - l_{\bold s, n}=\infty.$$
This follows easily from the following inequalities
\[
 \bold k_{\bold m, n} - l_{\bold s, n}=  \bold k_{\bold m, n} - l_{\bold s_{\ell_n}, n} \geq \eta_{\ell_n}(n) \geq \eta(n),
\]
and the fact that $\lim_{\mathcal D_{\bold m}} \eta(n)=\infty.$

Now fix $\ell_*<\omega$. Then for all $n$ with $\ell_n > \ell_*$ we have
\[
 \bold g_{\bold s_{\ell_*}, n} \unlhd \bold g_{\bold s_{\ell_n}, n}= \bold g_{\bold s, n},
\]
hence
\[
D_{\ell_*} \subseteq \{ n< \omega: \bold g_{\bold s_\ell, n} \unlhd \bold g_{\bold s, n}  \},
\]
where $D_{\ell_*}$ is as defined above. As $D_{\ell_*} \in \mathcal D_{\bold m}$, we have
$\{ n< \omega: \bold g_{\bold s_\ell, n} \unlhd \bold g_{\bold s, n}  \} \in \mathcal D_{\bold m}$, and hence
 $\bold s_\ell \leq_{\AP} \bold s.$

Finally (3)
follows from the way we defined the game $\Game^n_{k}(\bold m)$, noting that protagonist plays with a winning strategy.
\end{proof}
\begin{definition}
For each $\bold s \in \AP_{\bold m},$ we define $H_{\bold s}$ as the set
\[
H_{\bold s}=\bigg\{ (h_1, h_2): h_1 \in \prod\limits_n \MMB^{\bold m, 1}_n,~   h_2 \in \prod\limits_n \MMB^{\bold m, 2}_n \text{~and~}\{n<\omega: f_{\bold s, n}(h_1(n))=h_2(n)     \} \in \mathcal{D}_{\bold m}              \bigg\}.
\]
\end{definition}
The proof of the next claim follows from the way that we defined the game $\Game^n_k(\bold m)$ \ref{ehfrgame} and the fact that by Lemma \ref{eqv3} we have
$\mathcal D \supseteq \mathcal{D}_{\bold m}$ (see also the proof of lemmas \ref{ss1} and \ref{eqv3}).
\begin{claim}
\label{f200}
Let $\bold s \in \AP_{\bold m}.$ Then
\[
h_{\mathcal{D}, \bold s}=\{(h_1/\mathcal{D}, h_2/\mathcal{D}): (h_1, h_2) \in H_{\bold s}              \}
\]
defines a partial elementary mapping from $\prod\limits_{\mathcal D} \MMB^{\bold m, 1}_n$ into  $\prod\limits_{\mathcal{D}} \MMB^{\bold m, 2}_n$.
\end{claim}

Let $\mathcal F \subseteq \prod\limits_n [\omega]^{\bold k_{\bold m, n}}$
witness $\mathfrak{d}_{\bold m}=\aleph_1.$

For each $\ell \in \{1, 2\}$  let $\langle  \MMB^{\bold m, \ell}_{n, i}: i<\omega_1            \rangle$
be a $\prec$-increasing and continuous chain of countable elementary submodels of $\MMB^{\bold m, \ell}_{n}$ whose union is $\MMB^{\bold m, \ell}_{n}$.
For each $i<\omega_1$ and $\ell \in \{1, 2\}$, as $||\MMB^{\bold m, \ell}_{n, i}|| \leq \aleph_0$,
there is some
$\mathcal F^\ell_i \subseteq \prod\limits_n [\MMB^{\bold m, \ell}_{n, i}]^{\bold k_{\bold m, n}}$
of cardinality $\aleph_1$ such that
\[
\left(\forall \eta \in \prod\limits_n \MMB^{\bold m, \ell}_{n, i}\right)\left( \exists f  \in \mathcal F^\ell_i\right) \forall_{\mathcal{D}_{\bold m}}n \left(\eta(n) \in f(n)\right).
\]
Let $\langle f^\ell_j: j<\omega_1   \rangle$
enumerate $\bigcup\limits_{i<\omega_1} \mathcal F^\ell_i$. By induction on $i<\omega_1$, and using Claim \ref{f201},
we choose $\bold s_i$ such that:
\begin{enumerate}
\item $\bold s_i \in \AP_{\bold m},$

\item $i< j \implies \bold s_i \leq_{\AP_{\bold m}} \bold s_j$,

\item if $i=2j+1$, then
\[
n< \omega \text{~and~} l_{\bold s_{2j}} < \bold k_{\bold m, n} \implies f^1_j(n) \subseteq \dom(f_{\bold s_i, n}),
\]

\item if $i=2j+2$, then
\[
n< \omega \text{~and~} l_{\bold s_{2j+1}} < \bold k_{\bold m, n} \implies f^2_j(n) \subseteq \range(f_{\bold s_i, n}).
\]
\end{enumerate}
Now let $\mathcal D$ be a non-principal ultrafilter on $\omega$ which extends $\mathcal{D}_{\bold m}$. Let
\[
h=\bigcup\{h_{\mathcal D, \bold s_i}: i<\omega_1    \}.
\]
\begin{claim}
\label{f202}
$h$ is an isomorphism from $\prod\limits_{\mathcal{D}} \MMB^{\bold m, 1}_n$ onto $\prod\limits_{\mathcal{D}} \MMB^{\bold m, 2}_n$.
\end{claim}
\begin{proof}
By Claim \ref{f200}, each $h_{\mathcal{D}, \bold s_i}$ is a partial elementary embedding from $\prod\limits_{\mathcal{D}} \MMB^{\bold m, 1}_n$ into $\prod\limits_{\mathcal{D}} \MMB^{\bold m, 2}_n$, and furthermore, for $i<j<\omega_1$
we have $h_{\mathcal{D}, \bold s_i} \subseteq h_{\mathcal{D}, \bold s_j}$. By the construction of the sequence
$\langle  \bold s_i: i<\omega_1    \rangle$, we clearly have $\dom(h)=\prod\limits_{\mathcal{D}} \MMB^{\bold m, 1}_n$
and $\range(h)=\prod\limits_{\mathcal{D}} \MMB^{\bold m, 2}_n$.
So we are done.
\end{proof}
Lemma \ref{f207} follows.
\end{proof}
The following is an immediate corollary of Lemma \ref{f207}.
\begin{corollary}
\label{f20} Suppose $\bold m$ is an ultraproduct problem, $\mathfrak{d}_{\bold m}=\aleph_1$ and for each $n<\omega$, $\MMB^{\bold m, 1}_n \equiv \MMB^{\bold m, 2}_n$. If $\mathcal D$ is a non-principal ultrafilter on $\omega$, then $\prod\limits_{\mathcal D} \MMB^{\bold m, 1}_n \cong \prod\limits_{\mathcal D} \MMB^{\bold m, 2}_n$.
\end{corollary}
We now turn to the problem of controlling the size of $\mathfrak{d}_{\bold m}$'s.
We again prove a slightly stronger result.
  We first start with the following simple lemma.
  \begin{lemma}
  \label{a5}
Assume $\MA_\kappa$, and let $\bold k=\langle \bold k_n: n<\omega \rangle \in {}^\omega{\omega}$ be such that $\lim_{\mathcal{D}_{\bold k}} \bold k_n=\infty$. Let $A \subseteq {}^\omega{\omega}$ be of size $\leq \kappa$. Then there exists $f \in \prod\limits_{n<\omega}[\omega]^{\leq \bold{k}_n}$ such that for each $\eta \in A$, we have $\forall_{\mathcal{D}_{\bold k}}n \left(\eta(n)\in f(n)\right)$.
  \end{lemma}
  \begin{proof}
   Let $\MPB$ be the following covering forcing notion. The conditions are pairs $p=(k_p, f_p)$, where
  \begin{enumerate}
  \item[$(\alpha)$] $k_p < \omega,$

  \item[$(\beta)$] $f_p \in \prod\limits_{n<\omega}[\omega]^{\leq \bold{k}_n}$ and $\{|f_p(n)|: n<\omega     \}$
  is bounded.
  \end{enumerate}
  Given conditions $p, q \in \MPB,$ let $p \leq q$ ($q$ is stronger than $p$) if
  \begin{enumerate}
  \item[$(\gamma)$] $k_q \geq k_p,$
  \item[$(\delta)$] $f_q \restriction k_p = f_p \restriction k_p$,
  \item[$(\epsilon)$] $\forall_{\mathcal{D}_{\bold k}}n, f_q(n) \supseteq f_p(n)$.
  \end{enumerate}
  It is easily seen that the forcing notion $\MPB$ is c.c.c. Indeed let
  $A \subseteq \MPB$ be of size $\aleph_1.$
  By shrinking $A$ if necessary, we can assume that $k_p=k_*$ for some fixed $k_* < \omega$
  and all $p \in A.$ Furthermore, we can assume that $f_p(n)=f_q(n)$
  for all $p, q \in A$ and all $n< k_*.$
  We show that any two conditions in $A$ are compatible. Thus let $p, q \in A.$ Let $l<\omega$
  be such that
  \[
  \forall n< \omega \big( |f_p(n)|, |f_q(n)| < l \big).
  \]
  Let
  $E=\{ n< \omega: \bold k_n > 2l   \}$. Then $E \in \mathcal D_{\bold k}$. Define $r \in \MPB$ as $r=(k_r, f_r)$,
  where:
  \begin{enumerate}
\item   $k_r=k_*,$

\item for all $n< k_*, f_r(n)=f_p(n)=f_q(n)$,

\item for all $n \in E \setminus k_*, f_r(n)=f_p(n) \cup f_q(n)$.

\item for all $n \in \omega \setminus (E \cup k_*),$ $f_r(n)=f_p(n)$.
\end{enumerate}
  The condition $r$ is easily seen to be an extension of both of $p$ and $q$.

  For each $\eta \in A$, the set
  \[
  D_\eta = \{p \in \MPB:  \forall_{\mathcal D_{\bold k}} n, \eta(n) \in f_p(n)   \}
  \]
  is clearly dense in $\MPB$. Thus by $\MA_\kappa$ we can find a filter $G$
  such that $G \cap D_\eta \neq \emptyset$ for all $\eta \in A$.
  Then the function $f$, defined as
  $f(n)=\bigcup\limits_{p \in G}f_p(n)$, has the required properties.
  \end{proof}
\begin{proposition}
\label{f12} Suppose the $\text{GCH}$ holds and $\lambda> \cf(\lambda)=\aleph_1$. Then there exists a c.c.c. forcing notion
$\MPB$ of size $\lambda$ such that in $V^{\MPB}$, we have:
\begin{enumerate}
\item $2^{\aleph_0}=\lambda,$
\item if $\bold k \in {}^\omega{\omega}$ is such that $\limsup\limits_{n<\omega} \bold k_n=\infty,$ then $\mathfrak{d}_{\bold k}=\aleph_1$.
\end{enumerate}
\end{proposition}
\begin{proof}
Let $\langle  \lambda_i: i < \omega_1      \rangle$ be an increasing  sequence of regular cardinals cofinal in $\lambda$. Let $\MPB=\langle  \langle \MPB_i: i \leq \omega_1 \rangle, \langle  \dot{\MQB}_i: i<\omega_1    \rangle\rangle       $
be a finite support iteration of c.c.c. forcing notions such that:
\begin{itemize}
\item for each $i<\omega_1$, $|\MPB_i| = \lambda_i$,
\item for each $i<\omega_1, V[G_{\MPB_{i}}]\models$``Martin's axiom'',
\item $(2^{\aleph_0})^{V[G_{\MPB_{i}}]} = \lambda_{i}$.
\end{itemize}
It is easily seen, using Lemma \ref{a5}, that $\MPB=\MPB_{\omega_1}$ is as required. Indeed, suppose that $\bold k$ is as in clause (2). Pick $i_* < \omega_1$ such that $\bold k \in V[G_{\MPB_{i_*}}]$. For each $i_* \leq i < \omega_1$, pick, using Lemma \ref{a5}, some $f_i \in \prod\limits_{n<\omega}[\omega]^{\bold{k}_n} \cap V[G_{\MPB_{i+1}}]$ such that:
\[
\left(\forall \eta \in {}^\omega{\omega} \cap V[G_{\MPB_i}]\right) (\forall_{\mathcal{D}_{\bold k}}n) \left(\eta(n)\in f_i(n)\right).
\]
Then the family $\mathcal{F}=\{f_i: i_* \leq i < \omega_1     \}$
witnesses  $\mathfrak d_{\bold k}=\aleph_1$.
\end{proof}
\begin{remark}
By \cite{sh:f}, the conclusion of Lemma \ref{f12} also holds in the Sacks model.
\end{remark}
Putting the above results together we get the following theorem, which extends \cite[Theorem 1.2]{gol-sh}.
\begin{theorem}
\label{d11}
Assume the $\text{GCH}$ holds and $\lambda> \cf(\lambda)=\aleph_1$. Then there exists a c.c.c. forcing notion
$\MPB$ of size $\lambda$ such that in $V^{\MPB}$:
\begin{enumerate}
\item $2^{\aleph_0}=\lambda$,

\item for every ultraproduct problem $\bold m$, if $\mathcal D$ is a non-principal ultrafilter on $\omega$
 and if $\prod\limits_{\mathcal D} \MMB^{\bold m, 1}_n \equiv \prod\limits_{\mathcal D} \MMB^{\bold m, 2}_n$,
then $\prod\limits_{\mathcal D} \MMB^{\bold m, 1}_n \cong \prod\limits_{\mathcal D} \MMB^{\bold m, 2}_n$.
\end{enumerate}
\end{theorem}
\begin{proof}
By Lemmas \ref{eqv3}, \ref{f20} and \ref{f12}.
\end{proof}
The models appearing in a pseudo ultraproduct problem that we have considered so far
were of size $\aleph_0$ or $\aleph_1$. We now consider the same situation where the models can be finite as well.
\begin{definition}
\label{fin1}
(see \cite{sh448} and \cite[Chapter V]{sh:f})
Suppose $f, g \in{}^{\omega}(\omega+1 \setminus \{0, 1\})$ and $g \leq f.$ Then
\begin{center}
$\mathfrak d_{f, g}=\min\bigg\{|\mathcal F|: \mathcal F \subseteq \prod\limits_{n<\omega}[f(n)]^{< 1+g(n)} \text{~and~} (\forall \eta \in \prod\limits_{n<\omega} f(n))( \exists a \in \mathcal F) (\forall n, \eta(n) \in a(n))     \bigg\}.$
\end{center}
\end{definition}
Given $f \in{}^{\omega}(\omega+1 \setminus \{0, 1\})$, we define the notion of a (pseudo) $f$-ultraproduct problem as follows.
\begin{definition}
\label{fin2}
Suppose $f \in{}^{\omega}(\omega+1 \setminus \{0, 1\})$. Then
a (pseudo) $f$-ultraproduct problem
$
\bold{m}=\langle (\MMB^{\bold m, 1}_n, \MMB^{\bold m, 2}_n,  \tau_{\bold m, n}): n < \omega              \rangle$
is defined as in the notion of a (pseudo) ultraproduct problem (see definitions \ref{upp} and \ref{a2}), but we require  for each
$n<\omega$ and $\ell=1,2$, $||\MMB^{\bold m, \ell}_n || \leq f(n).$
\end{definition}
Given an $f$-ultraproduct problem $\bold m,$ the sequence $\bold k_{\bold m}$ and the filter $\mathcal D_{\bold m}$
 are defined as before. The proof of the next lemma is essentially the same as in the proof of Lemma \ref{f207}.
\begin{lemma}
\label{fin3}
Suppose $f \in{}^{\omega}(\omega+1 \setminus \{0, 1\})$,  $\bold m$ is an $f$-ultraproduct problem, $f \geq \bold k_{\bold m}$,  $\mathfrak{d}_{f, \bold k_{\bold m}}=\aleph_1$ and $\mathcal D \supseteq \mathcal D_{\bold m}$ is a non-principal ultrafilter on $\omega$.
Then  $\prod\limits_{\mathcal D} \MMB^{\bold m, 1}_n \cong \prod\limits_{\mathcal D} \MMB^{\bold m, 2}_n$.
\end{lemma}
The following is analogous to Theorem
\ref{d11} for $f$-ultraproduct problems.
\begin{theorem}
\label{fin4}
Assume the $\text{GCH}$ holds. Then there exists a cardinal and cofinality preserving forcing notion
$\MPB$ such that in $V^{\MPB}$:
\begin{enumerate}
\item $2^{\aleph_0} \geq \aleph_2$,

\item for every $f \in{}^{\omega}(\omega+1 \setminus \{0, 1\})$ and every $f$-ultraproduct problem $\bold m$ with $\bold k_{\bold m} \leq f$, if $\mathcal D$ is a non-principal ultrafilter on $\omega$
and if $\prod\limits_{\mathcal D} \MMB^{\bold m, 1}_n \equiv \prod\limits_{\mathcal D} \MMB^{\bold m, 2}_n$,
then $\prod\limits_{\mathcal D} \MMB^{\bold m, 1}_n \cong \prod\limits_{\mathcal D} \MMB^{\bold m, 2}_n$.
\end{enumerate}
\end{theorem}
\begin{proof}
By \cite{sh448}, there exists a cardinal and cofinality preserving forcing notion
$\MPB$ which forces the failure of the continuum and such that in $V^{\MPB}$, for each
$f \in{}^{\omega}(\omega+1 \setminus \{0, 1\})$ and each $f$-ultraproduct problem $\bold m$ with $\bold k_{\bold m} \leq f$,
we have $\mathfrak d_{f, \bold k_{\bold m}}=\aleph_1$. Now the result follows from Lemmas \ref{eqv3} and \ref{fin3}.
\end{proof}

\section{An impossibility result}
\label{keisler2}
In this section, we show that we cannot extend Theorem \ref{d11} to get a model of $\text{MA}+2^{\aleph_0} > \aleph_1$
in which for every ultraproduct problem $\bold m$ and every non-principal ultrafilter $\mathcal D$  on $\omega$
  if $\prod\limits_{\mathcal D} \MMB^{\bold m, 1}_n \equiv \prod\limits_{\mathcal D} \MMB^{\bold m, 2}_n$,
then $\prod\limits_{\mathcal D} \MMB^{\bold m, 1}_n \cong \prod\limits_{\mathcal D} \MMB^{\bold m, 2}_n$.
Indeed we prove the following stronger result. Recall that the bounding number is defined as
$$\mathfrak{b}=\min\{|\mathcal{F}|: \mathcal{F} \subseteq {}^\omega{\omega}\text{~and~} \forall f \in {}^\omega{\omega} \exists g \in \mathcal{F}, g \nleq^* f     \},$$
 where $\leq^*$ is the eventual domination order.
\begin{theorem}
\label{ss2} Assume $\mathfrak{b} > \aleph_1$. Then there exists an  ultraproduct problem $\bold m$  such that:
\begin{enumerate}
\item $\MMB^{\bold m, 1}_n \equiv \MMB^{\bold m, 2}_n$ for every $n<\omega,$

\item For every non-principal ultrafilter $\mathcal D$  on $\omega$, $\prod\limits_{\mathcal D} \MMB^{\bold m, 1}_n \ncong \prod\limits_{\mathcal D} \MMB^{\bold m, 2}_n$.
\end{enumerate}
\end{theorem}
\begin{proof}
We follow \cite{gol-sh}. For every $n<\omega$, let $\MMB^{\bold m, 1}_n=(\MQB, <)$ be the dense linear order of rational numbers and
$\MMB^{\bold m, 2}_n=(N, <_N)$ be a dense linear order of size $\aleph_1$ which has a point $a$ with $\cf(N_a, <_N)=\aleph_1,$
where $N_a=\{ d \in N: d <_N a    \}$.

Suppose by the way of contradiction that $\prod\limits_{\mathcal D} \MMB^{\bold m, 1}_n \cong \prod\limits_{\mathcal D} \MMB^{\bold m, 2}_n$,
for some non-principal ultrafilter $\mathcal D$  on $\omega$ and let
$\pi: \prod\limits_{\mathcal D} \MMB^{\bold m, 1}_n \cong \prod\limits_{\mathcal D} \MMB^{\bold m, 2}_n$ witness such an isomorphism.

For notational simplicity set
$M_*=\prod\limits_{\mathcal D} \MMB^{\bold m, 1}_n = (\MQB, <)^\omega/\mathcal D$
and $N_*=\prod\limits_{\mathcal D} \MMB^{\bold m, 2}_n= N^\omega/ \mathcal D$. Let
$a_*=[\langle a: n<\omega     \rangle]_{\mathcal D} \in N_*.$ By \cite[Claim 2.2]{gol-sh}, $\cf((N_*)_{a_*})=\aleph_1$,
and hence  $\cf((M_*)_{a_\dagger})=\aleph_1$ where $a_\dagger \in M_*$ is such that $\pi(a_\dagger)=a_*.$
By \cite[Claim 2.4]{gol-sh}, $\cf((M_*)_{b_\dagger})=\aleph_1$ for every $b_\dagger \in M_*,$
in particular $\cf((M_*)_{0_\dagger})=\aleph_1$ where $0_\dagger= [\langle 0: n<\omega     \rangle]_{\mathcal D} \in M_*.$
\begin{claim}
$\cf((M_*)_{0_\dagger}) \geq \mathfrak{b}$.
\end{claim}
\begin{proof}
Suppose $\kappa < \mathfrak{b}$ and let $\langle [f_i]_{\mathcal{D}}: i<\kappa           \rangle$ be an increasing sequence  in $M_*$
where for each $i<\kappa, [f_i]_{\mathcal{D}} <_{M_*} 0_\dagger.$ We may assume that $-1< f_i(n) < 0$ for every $n<\omega.$ For each $i<\kappa$
set
\[
g_i(n)=\min\{k<\omega:  f_i(n) < -\frac{1}{k}    \}
\]
Then $\mathcal{G}=\{g_i: i<\kappa\} \subseteq{}^{\omega}\omega$ and $|\mathcal{G}| \leq \kappa$. Thus we can find some $g: \omega \to \omega$
such that for all $i<\kappa, g_i \leq^* g.$ Define $f:\omega \to \MQB$ by
\[
f(n)=-\frac{1}{g(n)}.
\]
For any $i<\kappa$ we can find some $\ell_i<\omega$ such that $g(n) > g_i(n)$ for all $n > \ell_i$ and hence for any such $n$,
\[
f(n)= -\frac{1}{g(n)} > -\frac{1}{g_i(n)} > f_i(n).
\]
It follows that the sequence  $\langle [f_i]_{\mathcal{D}}: i<\kappa           \rangle$ is bounded from above by
$[f]_{\mathcal D} < 0_\dagger$. Thus $\cf((M_*)_{0_\dagger}) > \kappa$. As $\kappa < \mathfrak{b}$ was arbitrary,
$\cf((M_*)_{0_\dagger}) \geq \mathfrak{b}$ and we are done.
\end{proof}
But this implies that $\cf((M_*)_{0_\dagger}) \geq \mathfrak{b} > \aleph_1$ which contradicts $\cf((M_*)_{0_\dagger})=\aleph_1$.
The theorem follows.
\end{proof}
\begin{remark}
Martin's axiom implies $\mathfrak{b}=2^{\aleph_0}$, and hence the conclusion of Theorem \ref{ss2} holds under $\text{MA}+2^{\aleph_0}> \aleph_1$.
\end{remark}

\end{document}